\documentclass[a4paper,11 pt,reqno]{amsart}
\usepackage{a4,amsbsy,amscd,amssymb,amstext,amsmath,latexsym,oldgerm,enumerate,verbatim,graphicx,xy,mathtools}

\makeatletter                                       %De subsecties in de inhoudsopgave moeten
\def\l@subsection{\@tocline{2}{0pt}{2.5pc}{2.5pc}{}}%iets opgeschoven worden naar rechts
\makeatother                                        %een eventuele tweede regel moet gewoon
\makeatletter                                                  %Hoofdstukken moeten op de eerst
\def\chapter{\clearpage\thispagestyle{plain}\global\@topnum\z@ %beschikbare pagina beginnen
\@afterindenttrue \secdef\@chapter\@schapter}
\makeatother

\newtheorem{thmgl} {Theorem}    %globaal genummerd
\newtheorem{propgl}{Proposition}
\newtheorem{lemgl} {Lemma}
\newtheorem{cornn}{Corollary}

\theoremstyle{definition}

\newtheorem{remgl} {Remark}
\newtheorem{remsgl} [remgl]{Remarks}

\newcommand{\mf}{\mathfrak}
\newcommand{\mc}{\mathcal}
\newcommand{\mb}{\mathbb}

\newcommand{\nts}{\negthinspace}     %handig
\newcommand{\Nts}{\nts\nts}

\newcommand{\ov}{\overline}

\newcommand{\sm}{\setminus}         %verzamelingen

\newcommand{\ot}{\otimes}           %vectorruimten en modules

\newcommand{\Hom}{{\rm Hom}}        %Algebra algemeen
\newcommand{\Mor}{{\rm Mor}}

\newcommand{\Mat}{{\rm Mat}}

\newcommand{\Sym}{{\rm Sym}} %symmetrische groep

\newcommand{\tr}{{\rm tr}}

\newcommand{\g}{\mf{g}}
\newcommand{\h}{\mf{h}}
\let\ttie\t
\newcommand{\tie}[1]{{\let\t\ttie \ttie#1}}%\t requires a special treatment, because
\renewcommand{\t}{\mf{t}}  %it is defined recursively. This trick is due to Uwe L\"uck
\newcommand{\n}{\mf{n}}
\newcommand{\gl}{\mf{gl}}
\newcommand{\spl}{\mf{sl}}

\newcommand{\GL}{{\rm GL}}

\newcommand{\SL}{{\rm SL}}

\newcommand{\e}{\epsilon}
\newcommand{\ve}{\varepsilon}

\newcommand{\Stab}{{\rm Stab}}

\newcommand*\bmat[1]{\begin{bsmallmatrix}#1\end{bsmallmatrix}}

\makeatletter
\def\vcdots{\vbox{\baselineskip4\p@ \lineskiplimit\z@
\kern3\p@\hbox{.}\hbox{.}\hbox{.}\Nts\nts\kern3\p@}}
\makeatother

\xyoption{all}

\begin{document}

\title{Highest-weight vectors for the adjoint action of $\GL_n$ on polynomials, II}

\begin{abstract}
Let $G=\GL_n$ be the general linear group over an algebraically closed field $k$ and let $\g=\gl_n$ be its Lie algebra. Let $U$ be the subgroup of $G$ which consists of the upper unitriangular matrices. Let $k[\g]$ be the algebra of polynomial functions on $\g$ and let $k[\g]^G$ be the algebra of invariants under the conjugation action of $G$. For all weights
$\chi\in\mathbb Z^n$ with $\chi_{{}_2}\le0$ or $\chi_{{}_{n-1}}\ge0$ we give explicit bases for the $k[\g]^G$-module $k[\g]^U_\chi$ of highest weight vectors of weight $\chi$. We also give bases for the vector spaces $k[C]^U_\chi$ where $C$ is a nilpotent orbit closure.
This extends earlier results to a much bigger class of weights. To express our semi-invariants in terms of matrix powers we prove certain Cayley-Hamilton type identities.
\end{abstract}

\author[R.\ H.\ Tange]{Rudolf Tange}


\maketitle
\markright{\MakeUppercase{Highest weight vectors for the adjoint action of} $\GL_n$, II}

\section*{Introduction}
Let $\GL_n$ be the general linear group over an algebraically closed field $k$ and let $\gl_n$ be its Lie algebra.
Let $k[\gl_n]$ be the ring of polynomial functions on $\gl_n$, it is a $\GL_n$-module under the conjugation action.
In this paper we will give, for certain weights, explicit bases for the spaces of highest weight vectors in $k[\gl_n]$
as modules over the ring $k[\gl_n]^{\GL_n}$ of invariants. Our main result Theorem~\ref{thm.basis} extends Theorems~1 and
2 in \cite{T}. The method we use here is quite different: in \cite{T} we used evaluation at certain special
nilpotent matrices, in this paper we use certain Cayley-Hamilton type identities and a morphism from the nilpotent cone to
the variety of $n\times(n-1)$ matrices which intertwines the adjoint action of the upper triangular matrices with the left
regular action.

Assume $k=\mb C$. In \cite[Sect.~4]{T} a general construction is given to produce $k[\gl_n]^{\GL_n}$-module generators for the module
$k[\gl_n]^U_\lambda$ of highest weight vectors of weight $\lambda$. It amounts to applying a highest weight vector
$E_\lambda\in L_{\mb C}(\lambda)\subseteq\gl_n^{\ot t}$, $t$ depending on $\lambda$, to varying tuples of fundamental invariants.
Whenever this method yields a $k[\gl_n]^{\GL_n}$-module basis of $k[\gl_n]^U_\lambda$ it also yields a $k[\gl_n]^{\GL_n}$-module
basis of the whole $L_{\mb C}(\lambda)$-isotypic component of $k[\gl_n]$. This can be seen by replacing $E_\lambda$ by arbitrary basis
elements of $L_{\mb C}(\lambda)$. Since the highest weight vectors in this paper are in accordance with this construction
(see Remark~\ref{rems.invs}.3), this applies to all the weights we consider in this paper.
Finding tuples of fundamental invariants for which the above method yields a basis is for many weights an intriguing combinatorial
problem.

We briefly sketch some of the relevant background when $k=\mb C$. In \cite{Kos} Kostant showed that for any complex reductive group $G$ with Lie algebra
$\g$ the coordinate rings of the fibers of the adjoint quotient $\g\to\g\,/\nts/G$ are all isomorphic as $G$-modules to the space of harmonics $H$.
For one particular fiber, the nilpotent cone $\mc N$, we even have an isomorphism of graded $G$-modules $k[\mc N]\cong H$.
In \cite{Hes} a formula was given for the graded multiplicity of an irreducible $L(\chi)$ in $k[\mc N]$.

In the case of $\GL_n$ it is well-known that the graded multiplicity of an irreducible $L(\chi)$ in $k[\mc N]$ is given
by the Kostka polynomial $K_{\chi+r{\bf 1},r{\bf 1}}$, where $\bf 1$ is the all-one vector of length $n$ and $r$ is such
that $\chi+r{\bf 1}$ has all its components $\ge0$. Lascoux and Sch\"utzenberger defined the charge $c(T)$ of a semi-standard
tableau $T$ such that $K_{\lambda,\mu}(q)=\sum_Tq^{c(T)}$, where the sum is over all semi-standard tableau $T$ of shape $\lambda$
and weight $\mu$, see e.g. \cite[III.6]{Mac}.

Every dominant weight $\chi$ in the root lattice of $\GL_n$ can be written in the form $\chi=[\lambda,\mu]_n\stackrel{\rm def}{=}\lambda-\mu^{\rm rev}$,
where $\lambda$ and $\mu$ are partitions of the same number whose lengths add up to at most $n$ and $\mu^{\rm rev}$ denotes the reverse of $\mu$
considered as an $n$-tuple. Now let $\lambda$ and $\mu$ be partitions of the same number. Then the graded multiplicity of $L([\lambda,\mu]_n)$
in $k[\mc N]$ has a limit as $n\to\infty$, i.e. in a fixed degree the multiplicity is constant, independent of $n$, for $n$ sufficiently big.
See \cite{Stan} and \cite{B4}. In \cite{B1} formulas are given for the graded multiplicity of $L(\chi)$ in $k[\mc N]$ which separate the
dependence on $\lambda$ and $\mu$ and the dependence on $n$. Proofs of the results in \cite{B1} can be deduced from the arguments mentioned
there and the results in \cite{B2} and \cite{B4}. A generalisation of the theory of Kostka polynomials to arbitrary types is given
in \cite{B3}. The assumptions in \cite[Thm.~3.4]{B3} are now known to be unnecessary, see \cite[Thm.~2]{KLT}.

The paper is organised as follows. In Section~\ref{s.prelim} we set up the basic notation and recall some results concerning the adjoint
action of $\GL_n$ on $k[\gl_n]$. In Section~\ref{s.matid} we prove certain Cayley-Hamilton type matrix identities. In Section~\ref{s.semiinvs} we
prove our main result Theorem~\ref{thm.basis}. The main idea is to use a birational morphism $\varphi$ from the nilpotent cone to the variety
$\Mat_{n,n-1}$ of $n\times(n-1)$ matrices. The results from Section~\ref{s.matid} are used to relate a family of semi-invariants to another
family of semi-invariants which can be interpreted as pullbacks along $\varphi$ of certain well-known highest weight vectors in $k[\Mat_{n,n-1}]$.
In Section~\ref{s.semiinvs2} we give bases for the spaces of highest weight vectors in the coordinate rings of nilpotent orbit closures.
The morphism $\varphi$ above is an isomorphism between $\{A\in\mc N\,|\,d(A)\ne0\}$, $d\in k[\mc N]$ a certain $B$-semi-invariant,
and a special open subset of $\Mat_{n,n-1}$. In Section~\ref{s.d} we prove some properties of $d$ and the localisation $k[\mc N]^U[d^{-1}]$.
In particular, it is made clear how much $k[\mc N]^U$ is simplified by making $d$ invertible.

\section{Preliminaries}\label{s.prelim}

Throughout this paper $k$ is an algebraically closed field and $G=\GL_n$, $n\ge 2$, is the group of invertible $n\times n$-matrices.
Its natural module is $k^n$ and its Lie algebra is $\g=\gl_n$, the vector space of $n\times n$-matrices.
We denote the standard basis elements of $k^n$ by $e_1,\ldots,e_n$.
The set of nilpotent $n\times n$ matrices with entries in $k$ is called the {\it nilpotent cone} and it is denoted by $\mc N$.

The Borel subgroup of $G$ which consists of the invertible upper triangular matrices is denoted by $B$ and its unipotent
radical, which consists of the upper unitriangular matrices, by $U$. We denote by $H$ the maximal torus of $G$ which consist
of the invertible diagonal matrices. We will identify the character group of $H$ with $\mathbb Z^n$ by means of the isomorphism
which sends the character $S\mapsto S_{ii}$ of $H$ to the $i$-th standard basis element $\ve_i$ of $\mathbb Z^n$.
We will call the characters of $H$ {\it weights}. As is well-known, we have for any (rational) $H$-module $V$ a weight space
decomposition $V=\bigoplus_\chi V_\chi$, where $V_\chi=\{v\in V\,|\,\forall_{S\in H}S\cdot v=\chi(S)v\}$, and we call the
$\chi$ for which $V_\chi\ne 0$ the weights of $V$.

If $\chi$ is a weight, then we will also consider it as an element of $\mathbb Z^n$ and denote its components by
$\chi_{{}_1},\ldots,\chi_{{}_n}$. We will use additive notation for characters: $(\chi+\eta)(S)=\chi(S)\eta(S)$.
Any weight $\chi$ will also be considered as a character of $B$ by $\chi(SA)=\chi(S)$ for all $S\in H$ and $A\in U$.
A weight $\chi$ is called {\it dominant} if $\chi_{{}_1}\ge\cdots\ge\chi_{{}_n}$. For $\chi$ a weight we put
$|\chi|=\sum_{i=1}^n\chi_{{}_i}$. The {\it root lattice} is the set of weights with coordinate sum $0$.
We denote the all-zero and all-one tuple in $\mb Z^n$ by ${\bf 0}={\bf 0}_n$ and ${\bf 1}={\bf 1}_n$.

If $\lambda$ is a partition, then we denote its length by $l(\lambda)$ and we denote the sum of its parts by $|\lambda|$.
The partitions of length $\le n$ are also considered as weights by extending them with zeros to a tuple of length $n$.
The partition of length one with single part equal to $r$ is denoted
by $(r)$ and the partition of length $r$ with all parts equal to $1$ is denoted by $1^r$.

If $\chi$ is a dominant weight, then we denote the corresponding irreducible $\GL_n(\mathbb C)$-module of highest weight $\chi$
by $L_{\mb C}(\chi)$. As is well-known, the dual module $L_{\mb C}(\chi)^*$ is isomorphic to $L_{\mb C}(-\chi^{\rm rev})$,
where $\chi^{\rm rev}$ denotes the reversed tuple of $\chi$.

Let $K$ be a group. If $V$ is a $K$-module over $k$, then we denote the space of $K$-invariants by $V^K$.
If $V$ is an algebraic variety over $k$ we denote the algebra of regular functions on $V$ by $k[V]$.
If $K$ acts on $V$, then it also acts on $k[V]$ via $(g\cdot f)(x)=f(g^{-1}\cdot x)$ for $g\in K$ and $x\in V$.
If $V$ is a (rational) $G$-module and $\chi$ is a weight, then the elements of $V^U_\chi$ are called
{\it highest weight vectors of weight $\chi$} or {\it $B$-semi-invariants of weight $\chi$}.

Since $G$ acts on $\g$ and the nilpotent cone $\mc N$ by conjugation, it also acts on $k[\g]$ and $k[\mc N]$.
The algebra $k[\g]$ is a polynomial algebra in the matrix entry functions $x_{ij}$,
and we have $k[\g]^U_\chi\ne0$ if and only if $\chi$ is dominant and in the root lattice.
Furthermore, $\dim k[\mc N]^U_\chi=\dim L_{\mb C}(\chi)_{\bf 0}$ and $k[\g]^U_\chi$ is a free $k[\g]^G$-module of rank
$\dim L_{\mb C}(\chi)_{\bf 0}$. See \cite[Prop.~1]{T} for references and more explanation.

\section{Certain matrix identities}\label{s.matid}

For $m$ an integer $\ge0$ put $s_m(A)=\tr(\wedge^m(A))$ and $h_m(A)=\tr(S^m(A))$, where $\wedge^m(A)$ and
$S^m(A)$ are the $m$-th exterior and symmetric power of $A\in\gl_n$ and $\tr$ denotes the trace.
For $m$ an integer $\ge1$ put $p_m(A)=\tr(A^m)$. As is well-known $s_1,\ldots,s_n$ are algebraically
independent generators of $k[\gl_n]^{\GL_n}$. The same holds for $h_1,\ldots,h_n$.
Note that $s_0=h_0=1$, $s_1=h_1=p_1=\tr$, $s_n=\det$ and $s_m=0$ for $m>n$.
For $i,j\in\{1,\ldots,n\}$ denote the partial differentiation with respect to $x_{ij}$ by $\partial_{ij}$.
For brevity we will often denote an $n\times n$ matrix with entries $a_{ij}$ by $(a_{ij})_{ij}$ rather than
$(a_{ij})_{1\le i,j\le n}$\,.

The purpose of this section is to prove the matrix identities
\begin{align}
(-1)^m(\partial_{ji}s_{m+1})(A)_{ij}&=A^m+\sum_{k=1}^m(-1)^ks_k(A)A^{m-k}\,,\label{eq.main1}\\
(\partial_{ji}h_{m+1})(A)_{ij}&=A^m+\sum_{k=1}^mh_k(A)A^{m-k}\,,\label{eq.main2}\\
(\partial_{ji}p_{m+1})(A)_{ij}&=(m+1)A^m\,\label{eq.main3}
\end{align}
for $m\ge0$ and $A\in\gl_n$. We will use this in the proof of Theorem~\ref{thm.basis} in Section~\ref{s.semiinvs}
to show that one family of highest weight vectors coincides up to signs with another when restricted
to the nilpotent cone.

We give some comments on equation \eqref{eq.main1}.
For $m=n$ it is the Cayley-Hamilton identity.
For $m>n$ it follows by multiplying the Cayley-Hamilton identity by $A^{m-n}$.
Finally, the case $m=n-1$ can easily be deduced
from the Cayley-Hamilton identity by replacing $s_n(A)I$ by $A (\partial_{ji}s_n)(A)_{ij}$
(the second factor is just the ``cofactor matrix") and multiplying through by $A^{-1}$.

Note that equations~\eqref{eq.main1}, \eqref{eq.main2} are equivalent to the
recursive equations
\begin{align}
(\partial_{ji}s_{m+1})(A)_{ij}&=-A\,(\partial_{ji}s_m)(A)_{ij}+s_m(A)I\,,\label{eq.main1'}\\
(\partial_{ji}h_{m+1})(A)_{ij}&=A\,(\partial_{ji}h_m)(A)_{ij}+h_m(A)I\label{eq.main2'}
\end{align}
for $m\ge 0$.

By \cite[Sect.~I.1.2]{Mac} and the Chevalley Restriction Theorem we have the identity
\begin{equation*}
\sum_{k=0}^m(-1)^ks_kh_{m-k}=0\,,\eqno{(*)}
\end{equation*}
$m\ge1$, which says that for all $N\ge 0$ the matrices
$${((-1)^{i-j}s_{i-j})}_{0\le i,j\le N}\text{\quad and\quad}{(h_{i-j})}_{0\le i,j\le N}$$
are each others inverse, where we put $s_m=h_m=0$ for $m<0$.
From $(*)$ we deduce that equations~\eqref{eq.main1} and \eqref{eq.main2} are also equivalent to resp.
\begin{align}
(-1)^mA^m&=(\partial_{ji}s_{m+1})(A)_{ij}+\sum_{k=1}^m(-1)^kh_k(A)(\partial_{ji}s_{m+1-k})(A)_{ij}\text{\quad and}\label{eq.main1''}\\
A^m&=(\partial_{ji}h_{m+1})(A)_{ij}+\sum_{k=1}^m(-1)^ks_k(A)(\partial_{ji}h_{m+1-k})(A)_{ij}\,\label{eq.main2''}
\end{align}
for $m\ge0$.

\begin{thmgl}\label{thm.matid}
Equations ~\eqref{eq.main1}-\eqref{eq.main2''} hold.
\end{thmgl}
\begin{proof}
Because of the preceding remarks it suffices to prove \eqref{eq.main3}, \eqref{eq.main1'} and \eqref{eq.main2'}.
Note that both sides of these equations are $G$-equivariant morphisms from $\g$ to $\g$.
% The map f\mapsto df:\g\to\Mor(\g,\g^*)=k[\g]\ot\g^* is G-equivariant, so ds_m\in\Mor_G(\g,\g^*)
% Now use the G-equivariant isomorphism \g\cong\g^* given by the trace form (E_{ij} corresponds to x_{ji})
If $M$ is a $G$-module, then we have a ``generalised Chevalley restriction map" $\Mor_G(\g,M)\to\Mor_W(\h,M^H)$, where
$\h=\g^H$ is the vector space of diagonal matrices and $W$ is the symmetric group of degree $n$.
This map is injective by the density of the semisimple elements.\footnote{This map can be defined for any reductive group.
In \cite{Br} it is determined when this map is surjective.}

So it suffices to prove the identities
\begin{align*}
\partial_i\pi_{m+1}&=(m+1)x_i^m\,,\\
\partial_i\sigma_{m+1}&=-x_i\partial_i\sigma_m+\sigma_m\quad\text{and}\\
\partial_i\eta_{m+1}&=x_i\partial_i\eta_m+\eta_m
\end{align*}
for all $i\in\{1,\ldots,n\}$, where $x_i=x_{ii}$, $\partial _i=\partial_{ii}$ and $\pi_m$, $\sigma_m$ and $\eta_m$ are the power sum and the elementary and complete
symmetric functions of degree $m$. The first identity is obvious. To prove the second identity we simply observe that $x_i\partial_i\sigma_m$ is the sum of the monomials
in $\sigma_m$ that involve $x_i$ and $\partial_i\sigma_{m+1}$ is the sum of the monomials in $\sigma_m$ that don't involve $x_i$.
To prove the third identity we calculate in the notation of \cite{Mac}
$$x_i\partial_i\eta_m=x_i\sum_{|\alpha|=m,\alpha_i>0}\alpha_ix^{\alpha-\e_i}=\sum_{|\alpha|=m}\alpha_ix^\alpha\quad\text{and}$$
$$\partial_i\eta_{m+1}=\sum_{|\alpha|=m+1, \alpha_i>0}\alpha_ix^{\alpha-\e_i}=\sum_{|\alpha|=m}(\alpha_i+1)x^\alpha\,,$$
where $\alpha$ runs over $\mb N^n$. From this the third identity immediately follows.

\begin{comment}%old proof
We note that equation~\eqref{eq.main1''} follows from \eqref{eq.main2''} by differentiating (*) with $m$ replaced by $m+1$.
So, in view of our previous remarks, it suffices to prove \eqref{eq.main2} and \eqref{eq.main3}. We clearly may assume that $k=\mathbb C$.

First we prove \eqref{eq.main3}. We use the ring $k[\delta]$ of dual numbers, $\delta^2=0$.
We have $(A+\delta X)^{m+1}-A^{m+1}=\delta(XA^m+AXA^{m-1}+\cdots+A^mX)$.
We get $(d_Ap_{m+1})(X)=\tr(XA^m+AXA^{m-1}+\cdots+A^mX)=(m+1)\tr(A^mX)$, since the factors in the trace of a product can be cyclically
permuted without changing the value. So $\partial_{ji}p_{m+1}(A)_{ij}=(d_Ap_{m+1})(E_{ji})_{ij}=(m+1)\tr(A^mE_{ji})_{ij}=(m+1)A^m$.

Now we prove \eqref{eq.main2} by induction. Assume it holds for all $k<m$.
Differentiating the identity (see \cite[Sect.~I.1.2]{Mac})
\begin{equation}\label{eq.ph}
mh_m=\sum_{k=1}^mp_kh_{m-k}\,.
\end{equation}
with $m$ replaced by $m+1$, we get
\begin{align*}
(m+1)(\partial_{ji}&h_{m+1})(A)_{ij}=\\
\sum_{k=1}^{m+1}h_{m+1-k}(A)(\partial_{ji}p_k)(A)_{ij}&+\sum_{k=1}^{m+1}p_k(A)(\partial_{ji}h_{m+1-k})(A)_{ij}=\\
\sum_{k=1}^{m+1}kh_{m+1-k}(A)A^{k-1}&+\sum_{k=1}^m\sum_{l=0}^{m-k}p_k(A)h_{m-k-l}(A)A^l=\\
\sum_{k=0}^m(k+1)h_{m-k}(A)A^k&+\sum_{l=0}^{m-1}\sum_{k=1}^{m-l}p_k(A)h_{m-l-k}(A)A^l=\\
\sum_{k=0}^m(k+1)h_{m-k}(A)A^k&+\sum_{l=0}^{m-1}(m-l)h_{m-l}(A)A^l=\\
(m+1)\sum_{k=0}^m&h_{m-k}(A)A^k
\end{align*}
by \eqref{eq.main3}, the induction hypothesis and \eqref{eq.ph} with $m$ replaced by $m-l$.
\end{comment}
\end{proof}
\begin{cornn}
We have
$$(-1)^m(\partial_{ji}s_{m+1})(A)_{ij}=A^m=(\partial_{ji}h_{m+1})(A)_{ij}$$
for all $A\in \mc N$.
\end{cornn}
\begin{proof}
This follows from \eqref{eq.main1} and \eqref{eq.main2}, since $\mc N$ is the zero locus of $s_1,\ldots,s_n$ and
also of $h_1,\ldots,h_n$.
\end{proof}

\begin{remsgl}\hfil\break
1. The RHS of \eqref{eq.main1} and \eqref{eq.main2} is, up to multiplication by a constant, the same as
the RHS of \cite[(3.9),(3.11)]{IOP} at $q=1$, except that the sum is there up to $m-1$
rather than $m$. In view of this, the LHS of these equations at $q=1$ should be equal to
$(\partial_{ji}s_{m+1})(A)_{ij}-s_m(A)I=-A\,(\partial_{ji}s_m)(A)_{ij}$
and $(\partial_{ji}h_{m+1})(A)_{ij}-h_m(A)I=A\,(\partial_{ji}h_m)(A)_{ij}$,
up to multiplication by a constant. I do not know an easy direct proof of this.
The LHS of \eqref{eq.main1} and \eqref{eq.main2} is %much
simpler and also what is needed for our purposes.\\
2.\ By \cite[Thm.~1]{T} the $k[\g]^G$-module $k[\g]^U_{\e_1-\e_n}\cong \Hom_G(\spl_n,k[\g])$ is free
with basis $(\partial_{1n}s_2,\ldots,\partial_{1n}s_n)$. So the $k[\g]^G$-module $\Mor_G(\g,\spl_n^*)\cong\Hom_G(\spl_n,k[\g])$
is free with $(ds_2|_{\spl_n},\ldots,ds_n|_{\spl_n})$ as a basis. From this we immediately deduce that
the $k[\g]^G$-module $\Mor_G(\g,\g^*)\cong\Hom_G(\g,k[\g])$ is free with basis $(ds_1,\ldots,ds_n)$.
Combining this with equation~\eqref{eq.main1} we now deduce that the morphisms $A\mapsto A^i$,
$0\le i\le n-1$, form a basis of the $k[\g]^G$-module $\Mor_G(\g,\g)\cong\Mor_G(\g,\g^*)$.
The fact that these morphisms span also follows immediately from the much more general \cite[Prop.~6.5]{KP}
which is also valid in positive characteristic by \cite{Don2}.
\begin{comment}
2. The identity \eqref{eq.main1'} is equivalent to the identity
$$d_As_{m+1}(X)=-d_As_m(XA)+s_m(A)\tr(X)$$ for all $X,A\in\gl_n$,
where $d_A$ denotes the differential at $A$.
It seems that neither this identity nor \eqref{eq.main1'} are
easy to prove directly.
\end{comment}
\end{remsgl}

\section{semi-invariants}\label{s.semiinvs}

In this section we will prove our main result Theorem~\ref{thm.basis}.
The idea is to use the morphism $\varphi$ from Lemma~\ref{lem.varphi} below. Loosely speaking it intertwines the
conjugation action of $B$ on $\mc N$ and the left regular action of $B$ on $n\times(n-1)$ matrices.
The $U$-invariants for the left regular action on matrices are, as is well-known, certain special
elements from the bideterminant basis, see \cite{DRS}. Pulling these back
along the morphism $\varphi$ gives us the required basis for the spaces of
highest weight vectors for the adjoint action.

We now introduce some further notation for partitions and weights, and some notation for tableaux and bideterminants.
The {\it shape} of a partition $\lambda$ is the set $$\{(i,j)\,|\,1\le j\le\lambda_i, 1\le i\le l(\lambda)\}\,.$$
We sometimes identify a partition with its shape.
The {\it transpose} or {\it conjugate} $\lambda'$ of a partition $\lambda$ with shape $S$ is the partition
with shape $\{(i,j)\,|\,(j,i)\in S\}$.
For $\lambda,\mu\in\mb Z^n$ we put $[\lambda,\mu]=\lambda-\mu^{\rm rev}$ where $\mu^{\rm rev}$ is the reversed tuple of $\mu$.
It is easy to see that for any dominant weight $\chi\in\mb Z^n$ there exists unique partitions $\lambda$ and $\mu$ with
$l(\lambda)+l(\mu)\le n$ and $\chi=[\lambda,\mu]$.

Let $I$ be a set of integers. A {\it tableau of shape $\lambda$ with entries in $I$} is a function from
the shape of $\lambda$ to $I$. If $I$ is not explicitly given, then $I=\{1,\ldots,n\}$.
The notion of row and column of a tableau and the transpose of a tableau are defined in the same way as for matrices.
We define the {\it content} or {\it weight} of a tableau $T$ to be $\sum\ve_{T(i,j)}$, where we sum over the $(i,j)$
in the shape of $\lambda$. So the $i^{\rm th}$ component of the content of $T$ is the number of times that $i$ occurs in $T$.
We say that a tableau is {\it semi-standard} if its entries are strictly
increasing down the columns and weakly increasing in the rows from left to right.
Elsewhere in the literature the term standard is used rather than semi-standard
and in \cite{Mac} all tableaux are semi-standard and standard has another meaning.
If $\lambda$ and $\mu$ are partitions, then a {\it rational tableau} of shape $(\lambda,\mu)$ with entries in $\{1,\ldots,m\}$ is a pair of
tableaux $(S,T)$ with entries in $\{1,\ldots,m\}$ where $S$ has shape $\lambda$ and $T$ has shape $\mu$.
If $S$ has weight $\nu$ and $T$ has weight $\eta$, then the rational tableau $(S,T)$ has weight $\nu-\eta$.
A rational tableau with entries in $\{1,\ldots,m\}$ is called semi-standard if
$S$ and $T$ are semi-standard and $|\{j\in S^1\,|\,j\le i\}\,|+|\{j\in T^1\,|\,j\le i\}\,|\le i$ for all $i\in\{1,\ldots,m\}$,
where $S^1$ and $T^1$ denote the first columns of $S$ and $T$. Rational tableaux are a convenient
combinatorial tool to deal with representations of $\GL_n$ that are not polynomial.
they are not related to the bideterminants below.

Denote the variety of $n\times m$-matrices by $\Mat_{n,m}$.
If $S=\bmat{s_1\vspace{-.1cm}\\\vdots\\s_l}$, $s_i\in\{1,\ldots,n\}$ and $T=\bmat{t_1\vspace{-.1cm}\\\vdots\\t_l}$,
$t_i\in\{1,\ldots,m\}$, are one-column tableaux of the same length $l$, then we define
$(S\,|\,T)\in k[\Mat_{n,m}]=k[(x_{ij})_{1\le i\le n,1\le j\le m}]$ to
be the determinant
\begin{equation*}
\det\big( (x_{s_i\,t_j})_{1\le i,j\le l}\big)\, .
\end{equation*}
If $S$ and $T$ are arbitrary tableaux of the same shape $\lambda$ with entries in $\{1,\ldots,n\}$ and $\{1,\ldots,m\}$
and with columns $S^i$ and $T^i$, then we define the
{\it bideterminant} $(S\,|\,T)$ to be the product of the determinants $(S^i\,|\,T^i)$, $1\le i\le\lambda_1$.
Recall that $\GL_n\times\GL_m$ acts on $k[\Mat_{n,m}]$ by $\big((P,Q)\cdot f\big)(A)=f(P^{-1}AQ)$.
If $S$ has weight $\mu$ and $T$ has weight $\nu$, then $(S\,|\,T)$ is an $H\times  H$-weight vector of weight $(-\mu,\nu)$.
In \cite{DeCEP} bideterminants are associated to pairs of rows rather than pair of columns. So
our bideterminant $(S\,|\,T)$ is equal to their bideterminant $(S'\,|\,T')$, where $S'$ and $T'$ denote the
transpose of $S$ and $T$. A tableau is semi-standard in our sense if and only if it its transpose is standard
in the sense of \cite{DeCEP}. Finally, we define, for $\lambda$ a partition, the {\it anti-canonical tableau}
$T_\lambda$ of shape $\lambda$ to be the tableau whose $i$-th column consists of the integers
$n-\lambda'_i+1,n-\lambda'_i+2,\ldots,n$. Its weight is $-\lambda^{\rm rev}$.

Now we introduce four families of highest weight vectors in $k[\gl_n]$ for the conjugation action of $\GL_n$.
For $T=\bmat{t_1\vspace{-.1cm}\\\vdots\\t_l}$ a one-column tableau of length $l$ with entries in $\{1,\ldots,n-1\}$
we define the following functions on $\gl_n$.
\begin{equation*}
\begin{split}
u_T&=\det\big((\partial_{1i}s_{t_j+1})_{n-l+1\le i\le n,\, 1\le j\le l}\big),\\
v_T&=\det\big((\partial_{in}s_{t_j+1})_{1\le i\le l,\, 1\le j\le l}\big),\\
\tilde u_T(A)&=\det\big(A^{t_1}(e_1)|\cdots|A^{t_l}(e_1)\big)_{l\ge}\text{\quad and}\\
\tilde v_T(A)&=\det\big(A'^{t_1}(e_n)|\cdots|A'^{t_l}(e_n)\big)_{l\le}\,,
\end{split}
\end{equation*}
where the $s_i$ and $\partial_{ij}$ are as defined at the beginning of Section~\ref{s.matid}, the subscripts ``$l\ge$" and ``$l\le$" mean
that we take the last resp. first $l$ rows, and $A'$ denotes the transpose of $A$.
If $T$ has a repeated entry, then the elements $u_T$ and $v_T$ are zero. Otherwise they are up to sign equal to the elements
$u_{l,\{t_1+1,\ldots,t_l+1\}}$ and $v_{l,\{t_1+1,\ldots,t_l+1\}}$ from \cite[Thm.~1(ii)]{T}.
From the proof of \cite[Thm.~1(ii)]{T} it now follows that $u_{T}$ is a homogeneous $B$-semi-invariant of degree $\sum_{i=1}^lt_i$ and weight
$[(l),1^l]$.\footnote{The argument in the proof of \cite[Thm.~1(ii)]{T} showing that the $u_{t,I}$ are $B$-semi-invariants of weight
$\lambda^t$ should have appeared at the beginning of that proof.}
A simple computation shows that the same holds for $\tilde u_{T}$.
Similarly, $v_{T}$ and $\tilde v_T$ are homogeneous $B$-semi-invariants of degree $\sum_{i=1}^lt_i$ and weight $[1^l,(l)]$.

If $T$ is an arbitrary tableau of shape $\lambda$ with entries in $\{1,\ldots,n-1\}$ and with columns $T^i$, then we define $u_T$, $v_T$, $\tilde u_T$ and
$\tilde v_T\in k[\gl_n]$ to be the product of, respectively, the determinants $u_{T^i}$, $v_{T^i}$, $\tilde u_{T^i}$ and $\tilde v_{T^i}$, $1\le i\le\lambda_1$. By the above $u_T$ and $\tilde u_T$ are homogeneous $B$-semi-invariants of degree
$|T|=\sum_{(i,j)}T(i,j)$, where we sum over the shape of $\lambda$, and weight $[(|\lambda|),\lambda]$.
Similarly, $v_{T}$ and $\tilde v_T$ are homogeneous $B$-semi-invariants of degree $|T|$ and weight $[\lambda,(|\lambda|)]$.

Finally, if $T$ is the one-column tableau ${\bmat{1\vspace{-.1cm}\\\vdots\\n-1}}$, then $u_T=\pm\tilde u_T$ by equation \eqref{eq.main1}.
We denote this $u_T$ by $d$. For general $u_T$ and $\tilde u_T$ we only have such an equality as functions on $\mc N$.
By the above, $d$ has degree $\frac{1}{2}n(n-1)$ and weight $[(n-1),1^{n-1}]$.

In the remainder of this section we consider the nilpotent cone $\mc N$ as a $\GL_n$-variety via conjugation and we consider
the variety of $n\times(n-1)$ matrices $\Mat_{n,n-1}$ as a $\GL_n$-variety via left multiplication.

\begin{lemgl}\label{lem.varphi}
Let $\varphi:\mc N\to\Mat_{n,n-1}$ be the morphism $A\mapsto\big(A(e_1)|\cdots|A^{n-1}(e_1)\big)$.
Then the following hold.
\begin{enumerate}[{\rm(i)}]
\item The morphism $\varphi$ is birational.
\item $\varphi(SAS^{-1})=S\varphi(A)$ for all $S\in\Stab_{G}(e_1)$ and all $A\in\mc N$, and
if $f\in k[\Mat_{n,n-1}]$ is an $H$-weight vector of weight $\chi$,
then $f\circ\varphi\in k[\mc N]$ is an $H$-weight vector of weight $\chi-|\chi|\ve_1$.
\end{enumerate}
\end{lemgl}
\begin{proof}
(i).\ In the proof below we will temporarily use the letter $B$ to denote an $n\times(n-1)$ matrix rather than the group of invertible upper triangular matrices.
Let $\ov d$ be the minor $\Big(\bmat{2\vspace{-.1cm}\\\vdots\\n}|\bmat{1\vspace{-.1cm}\\\vdots\\n-1}\Big)$ on $\Mat_{n,n-1}$.
Then $d|_{\mc N}=\pm\ov d\circ\varphi$. We show that $\varphi$ is an isomorphism between the special open subsets
$$U_d=\{A\in\mc N\,|\,d(A)\ne0\}\text{\ and\ }V_{\ov d}=\{B\in\Mat_{n,n-1}\,|\,\ov d(B)\ne0\}\,.$$
Clearly, $U_d=\varphi^{-1}(V_{\ov d})$.
If $B\in V_{\ov d}$, then $[e_1|B]$ is invertible. Now let $N$ be the $n\times n$ matrix which
is $1$ on the first lower co-diagonal and zero elsewhere, and define $\psi:V_{\ov d}\to\mc N$
by $\psi(B)=[e_1|B]N[e_1|B]^{-1}$. Then $\varphi(\psi(B))=B$, since $[e_1|B]$ fixes $e_1$. In particular $\psi(B)\in U_d$.
If $A\in U_d$, then $\big(e_1,A(e_1),\ldots,A^{n-1}(e_1)\big)$ is a linear basis of $k^n$ and $[e_1|\varphi(A)]$ sends $e_i$ to $A^{i-1}e_1$.
So $\psi(\varphi(A))=A$.\\
(ii).\ The first formula is obvious. Now let $f$ be as stated, let $A\in\mc N$ and let $S\in H$. Then
$f(\varphi(S^{-1}AS))=f(s_{11}S^{-1}\varphi(A))=s_{11}^{-|\chi|}\chi(S)f(\varphi(A))$, since $f$ is homogeneous
of degree $-|\chi|$.
\end{proof}

\begin{thmgl}\label{thm.basis}
Let $\lambda$ be a partition of $r$ with length $\le n-1$.
\begin{enumerate}[{\rm(i)}]
\item The $u_T$, $T$ semi-standard of shape $\lambda$ with entries in $\{1,\ldots,n-1\}$, form a basis of the $k[\g]^G$-module $k[\g]^U_{[(r),\lambda]}$.
The same holds for the $\tilde u_T$.
\item The $v_T$, $T$ semi-standard of shape $\lambda$ with entries in $\{1,\ldots,n-1\}$, form a basis of the $k[\g]^G$-module $k[\g]^U_{[\lambda,(r)]}$.
The same holds for the $\tilde v_T$.
\end{enumerate}
\end{thmgl}
\begin{proof}
(i).\ To prove the first assertion it suffices, by \cite[Lem.~2, Prop.~1]{T} (essentially the graded Nakayama Lemma), to show that the restrictions of
the $u_T$ to $\mc N$, $T$ as stated, form a basis of $k[\mc N]^U_{[(r),\lambda]}$ as a vector space over $k$. Similarly for the second assertion and
the $\tilde u_T$. By the corollary to Theorem~\ref{thm.matid} $\tilde u_T=\pm u_T$ on $\mc N$, so it suffices to prove the assertion about the restrictions
of the $\tilde u_T$ to $\mc N$.

Let $\varphi$ be the morphism from Lemma~\ref{lem.varphi}. Then $\tilde u_T|_{\mc N}=(T_\lambda\,|\,T)\circ\varphi$. By \cite[Thm.~3.3]{DeCEP} the
bideterminants $(T_\lambda\,|\,T)$, $T$ semi-standard of shape $\lambda$ with entries in $\{1,\ldots,n-1\}$ form a basis of $k[\Mat_{n,n-1}]^U_{-\lambda^{\rm rev}}$.
By Lemma~\ref{lem.varphi}(i) the comorphism of $\varphi$ is injective, so the $\tilde u_T|_{\mc N}$, $T$ as above, are linearly independent.
Furthermore, $\tilde u_T\in k[\mc N]^U_{[(r),\lambda]}$ by Lemma~\ref{lem.varphi}(ii).

We can now finish with a dimension argument. By \cite[Thm.~11]{Kos} (see \cite[Prop.~1]{T} for references for the case of prime characteristic)
$k[\mc N]^U_{\chi}$ has dimension equal to $\dim L_{\mathbb C}(\chi)_{\bf 0}$, where $\chi=[(r),\lambda]$. By \cite[Prop.~2.4(b)]{Stem} this is equal to
the number of semi-standard rational tableaux of shape $((r),\lambda)$ and weight $\bf 0$. The extra condition for semi-standardness of a rational tableau
means in this case that the two tableaux should not contain $1$, since the two tableaux must have the same weight.
So these rational tableaux are clearly in bijection with the semi-standard tableaux of shape $\lambda$ with entries
in $\{1,\ldots,n-1\}$ (of arbitrary weight): simply lower all entries of the second tableau by $1$
and omit the first tableau (of shape $(r)$). Alternatively, this can be deduced from \cite[Ex.~I.6.2(c)]{Mac} as follows. By tensoring with $\det^{\lambda_1}$
we get that $\dim L_{\mathbb C}(\chi)_{\bf 0}$ is equal to the Kostka number $K_{\chi+\lambda_1{\bf 1},\lambda_1{\bf 1}}$.
By \cite[Ex.~1.6.2(c)]{Mac} we have that  this is equal to
$\dim L^{\GL_{n-1}}_{\mathbb C}(\lambda_1{\bf 1}_{n-1}-\lambda^{{\rm rev}_{n-1}})=\dim L^{\GL_{n-1}}_{\mathbb C}(-\lambda^{{\rm rev}_{n-1}})$,
where $\lambda^{{\rm rev}_{n-1}}$ is the reverse of $\lambda$ as an $(n-1)$-tuple. Since dual modules have the same dimension, this is equal
to $\dim L^{\GL_{n-1}}_{\mathbb C}(\lambda)$, which is the number of semi-standard tableaux of shape $\lambda$ with entries in $\{1,\ldots,n-1\}$.\\
(ii).\ Let $\Phi$ be the involution of the algebra $k[\g]$ corresponding to the involution $A\mapsto PA'P$ of the vector space $\g=\gl_n$, where
$A'$ is the transpose of $A$ and $P$ is the permutation matrix which is one on the anti-diagonal and zero elsewhere.
Then $\Phi(k[\g]^U_{\chi})=k[\g]^U_{-\chi^{\rm rev}}$, $\Phi(u_T)=\pm v_T$ and $\Phi(\tilde u_T)=\pm\tilde v_T$. So (ii) follows from (i).
\end{proof}

\begin{remsgl}\label{rems.invs}
1.\ The semi-invariants $\tilde u_T$ and $\tilde v_T$ also occur in \cite{Dom} and, in much bigger generality, in \cite{BR}.\\
2.\ Assume $k=\mb C$. Let $\lambda$ be a partition of $r$ with $l(\lambda)\le n-1$. By \cite[Prop.~4.2(ii)]{B1} the graded multiplicity of $L([(r),\lambda])$
in $k[\mc N]$ is given by $s_{\lambda}(q,\ldots,q^{n-1})$, where $s_{\lambda}$ is the Schur function associated to $\lambda$.
In accordance with Theorem~\ref{thm.basis}(i), this is equal to $\sum_Tq^{|T|}$, where the sum is over the semi-standard tableaux of
shape $\lambda$ with entries in $\{1,\ldots,n-1\}$. Note that there is a unique semi-standard tableau of shape $\lambda$ with $|T|$ minimal,
it has all entries in the $i$-th  row equal to $i$. So the first degree where $L([(r),\lambda])$ occurs in
$k[\g]$ or $k[\mc N]$ is $\sum_{i=1}^{l(\lambda)}i\lambda_i$, and in this degree it occurs with multiplicity $1$.
The same is true for $L([\lambda,(r)])$. The highest degree where $L([(r),\lambda])$ occurs in $k[\mc N]$ was determined by Kostant for any
reductive group. See \cite[Thm.~17]{Kos}. In our case this is $\sum_{i=1}^{l(\lambda)}(n-i)\lambda_i$.\\
3.\ Let $\chi=[\lambda,\mu]$ be dominant and in the root lattice, then $\lambda$ and $\mu$ are partitions of the same number, $t$ say.
In \cite[\S2]{T} a natural highest weight vector $E_\chi\in\g^{\ot t}$ was defined. Let $\psi_t$ and $\vartheta$ be as defined there.
In \cite[\S4]{T} it was asked whether the elements
\begin{equation}\label{eq.spset}
\vartheta\big(\psi_t(E_\lambda)\cdot s_{i_1}\ot\cdots\ot s_{i_t}\big)\,,
\end{equation}
$2\le i_1,\ldots,i_t\le n$,
generate the $k[\g]^G$-module $k[\g]^U_{\chi}$. In view of Theorem~\ref{thm.basis} the answer to this question is affirmative if $\chi_{{}_2}\le0$
or $\chi_{{}_{n-1}}\ge0$. One only has to observe that if $\chi_{{}_2}\le0$ the column stabiliser $C_\lambda$ is trivial, so each of the above
semi-invariants is a sum over $C_\mu$ which factorises as a product. Each of the factors in this product is equal to a $u_T$, $T$ a one-column
tableau. This also makes clear that Theorem~\ref{thm.basis} generalises both Thm.~1 and Thm.~2 in \cite{T}. So the statement at the end of
\cite{T} that the invariants from \cite[Thm.~2]{T} are not formed in accordance with the question in \S4 there is incorrect.\\
4.\ Put $\chi_t=[1^t,1^t]$. By \cite[Prop.~4.2(i)]{B1} the graded multiplicity of $L_{\mb C}(\chi_t)$ in $k[\mc N]$ is
$\frac{q^t(1-q^{n-2t+1})}{1-q^{n-t+1}}{n\brack t}_q$ which tends to $\frac{q^t}{(1-q)\cdots(1-q^t)}$ as $n\to\infty$.
%{n\brack t}_q=\prod_{i=0}^{t-1}(1-q^{n-i})/\prod_{i=0}^{t-1}(1-q^{t-i})
So the stable multiplicity (see \cite{B4} for more about this) of $L_{\mb C}(\chi_t)$ in the degree $d$-piece of $k[\mc N]$
is the number of partitions of $d$ of length $t$. For the weight $\chi_t$ changing the order of the arguments in \eqref{eq.spset} gives the same element
up to a sign. So \eqref{eq.spset} gives us a candidate spanning set labelled by partitions of length $t$ (lower all indices by $1$) and
by the above the elements of this set labelled by partitions of degree $d$ and length $t$ would have to be independent for big $n$.\\
5.\ Computer calculations show that the answer to the question from \cite[\S4]{T} (Remark~3 above) is in general no. For example, for $n=4$ and $\chi=(2,1,-1,-2)$ the
elements from \eqref{eq.spset} give only a $1$-dimensional space of invariants in degree $6$ and not the required $2$ dimensions.
However, it turns out that in all cases were I found that the elements from \eqref{eq.spset} don't span, replacing $E_\lambda$ by a suitable
$\Sym_t\times\Sym_t$-conjugate does give a spanning set.
\end{remsgl}

\section{Nilpotent orbit closures}\label{s.semiinvs2}
In this section we will extend results from Section~\ref{s.semiinvs} to nilpotent orbit closures. Recall that nilpotent orbits in $\gl_n$ are parameterised by partitions of $n$, see \cite{Jan2}. We denote the nilpotent orbit corresponding to the partition $\mu$ by $\mc O_\mu$. In the Jordan normal form of an element of $\mc O_\mu$ the blocks have sizes $\mu_1,\ldots,\mu_m$, where $m$ is the length of $\mu$. It turns out that for the weights that we consider only the biggest block size $\mu_1$ is relevant.

\begin{thmgl}\label{thm.basis2}
Let $C=\ov{\mc O}_\mu$ be a nilpotent orbit closure and let $\lambda$ be a partition of $r$ with length $l(\lambda)\le n-1$. Then the vector space $k[C]^U_{[(r),\lambda]}$ is nonzero if and only if $l(\lambda)\le\mu_1-1$. The same holds for the vector space $k[C]^U_{[\lambda,(r)]}$. Furthermore the following hold.
\begin{enumerate}[{\rm(i)}]
\item The $\tilde u_T|_C$, $T$ semi-standard of shape $\lambda$ with entries in $\{1,\ldots,\mu_1-1\}$, form a basis of $k[C]^U_{[(r),\lambda]}$.
\item The $\tilde v_T|_C$, $T$ semi-standard of shape $\lambda$ with entries in $\{1,\ldots,\mu_1-1\}$, form a basis of $k[C]^U_{[\lambda,(r)]}$.
\end{enumerate}
\end{thmgl}
\begin{proof}
We only prove (i), the proof of (ii) is completely analogous. Put $m=\mu_1-1$. By \cite[Thm.~2.1(c), Lem.~1.3(a)]{Don}, $(\mc N,C)$ is a good pair of $G$-varieties. So, by standard properties of good filtrations, the restriction of functions $k[\mc N]\to k[C]$ induces a surjection $k[\mc N]^U_{[(r),\lambda]}\to k[C]^U_{[(r),\lambda]}$. Furthermore, it is clear from the definition of $\tilde u_T$ that $\tilde u_T$ is zero on $C$ unless $T$ has all its entries in $\{1,\ldots,m\}$. In particular,
we must have $l(\lambda)\le m$.

So, by Theorem~\ref{thm.basis}, it suffices to show that the $\tilde u_T$ where $T$ is semi-standard of shape $\lambda$ with entries in $\{1,\ldots,m\}$ are linearly independent on $C$. Let $\varphi_m:C\to\Mat_{n,m}$ be the morphism given by $\varphi_m(A)=(A(e_1)|\cdots|A^m(e_1))$. Then the analogue of Lemma~\ref{lem.varphi}(ii) for $\varphi_m$ holds and $\tilde u_T|_C=(T_\lambda\,|\,T)\circ\varphi_m$.
So, because of the arguments in the proof of Theorem~\ref{thm.basis}, it suffices to show that the comorphism of $\varphi_m$ is injective, i.e. that $\varphi_m$ is dominant. Let $V$ be the nonempty open subset of $\Mat_{n,m}$ consisting of the matrices which after removing the first row still have maximal rank $m$. Let $B=(v_1|\cdots|v_m)\in V$. Then $(e_1,v_1,\ldots,v_m)$ is independent, so we can extend it to a basis of $k^n$. Now define $A\in\g$ by $A(v_i)=v_{i+1}$ for all $i\in\{0,\cdots,m\}$, where we put $v_0=e_1$ and $v_{m+1}=0$, and $A$ is zero on the other basis vectors. Then $A\in\mc O_{\mu_11^{n-\mu_1}}\subseteq C$ and $\varphi_m(A)=B$. So indeed $\varphi_m$ is dominant.
\end{proof}

\begin{remsgl}
1.\ Assume $k=\mb C$. Let $\mu$, $\lambda$ and $C$ be as in Theorem~\ref{thm.basis2} and assume $l(\lambda)\le\mu_1-1$. Then the graded multiplicity of $L([(r),\lambda])$
in $k[C]$ is equal to $\sum_Tq^{|T|}=s_\lambda(q,\ldots,q^{\mu_1-1})$, where the sum is over the semi-standard tableaux of shape $\lambda$ with entries in $\{1,\ldots,\mu_1-1\}$.
As in the case of $k[\mc N]$ (Remark~\ref{rems.invs}.2), the first degree where $L([(r),\lambda])$ occurs in $k[C]$ is $\sum_{i=1}^{l(\lambda)}i\lambda_i$, and in this degree it occurs with multiplicity $1$. There is also a unique tableau with $|T|$ maximal. It is obtained by filling the $i$-th row with $\mu_1-i$'s and then reversing all the columns. So the highest degree where $L([(r),\lambda])$ occurs in $k[C]$ is $\sum_{i=1}^{l(\lambda)}(\mu_1-i)\lambda_i$, and in this degree it occurs with multiplicity $1$. The same is true for $L([\lambda,(r)])$.\\
2.\ Let $C\subseteq\mc N$ be a nilpotent orbit closure. By \cite[Cor.~2.1(e)]{Don} the graded formal character of $k[C]$ is independent of the field $k$.
It is given by the ``Hesselink-Peterson-type" formula \cite[8.17]{Jan2}.
\end{remsgl}

\section{Further properties of the element $d$}\label{s.d}

We will now denote by $d$ the restriction to $\mc N$ of the $B$-semi-invariant $d$ from Section~\ref{s.semiinvs}.
Recall that $d$ has weight $[(n-1),1^{n-1}]$.
We will prove in this section some further properties of the element $d$ and the localisation $k[\mc N]^U[d^{-1}]$.
Note that $B$ acts rationally on $k[\mc N][d^{-1}]$. Recall that $\mc N$ is a normal variety, see e.g. \cite[8.5]{Jan2}.

\begin{propgl}\label{prop.d}
Denote the weight $[(n-1),1^{n-1}]$ of $d$ by $\eta$.
\begin{enumerate}[{\rm(i)}]
\item ${\rm CL}(\mc N)=\mathbb Z/n\mathbb Z$.
\item The variety $D\stackrel{\rm def}{=}\{A\in\mc N\,|\,d(A)=0\}$ is irreducible and
${\rm CL}(\mc N)$ is generated by the class of $D$. The principal divisor $(d)$ is equal to $nD$.
\item Let $\chi$ be a dominant weight in the root lattice. Then the embedding $k[\mc N]^U_\chi\Nts\to k[\mc N]^U_{\chi+\eta}$
given by multiplication by $d$ is an isomorphism if $\chi_{{}_2}\le 0$.
\item We have $k[\mc N][d^{-1}]^U_\chi\ne0$ if and only if there exists an $r\ge0$ such that $\chi+r\eta$ is dominant and in the root lattice
and has second component $\le0$. Now let $r$ be like that and let $\lambda$ be the partition given by
$[(s),\lambda]=\chi+r\eta$, where $s=|\lambda|$. Then $\dim k[\mc N][d^{-1}]^U_\chi$ is equal to the number of semi-standard tableaux
of shape $\lambda$ with entries in $\{1,\ldots,n-1\}$.
\end{enumerate}
\end{propgl}
\begin{proof}
(i).\ This is no doubt well-known; we include a proof for lack of reference.
If $\mc O_{\rm reg}$ is the regular nilpotent orbit, then $\mc N\sm\mc O_{\rm reg}$ has codimension $\ge 2$, so
${\rm CL}(\mc N)={\rm CL}(\mc O_{\rm reg})={\rm Pic}(\mc O_{\rm reg})$. Now let $N$ be the $n\times n$ matrix which
is $1$ on the first lower co-diagonal and zero elsewhere. Then $N\in\mc O_{\rm reg}$ and $\mc O_{\rm reg}\cong G/K$, where
$K={\rm Stab}_G(N)$. Since $K$ is the centre $Z$ of $G=\GL_n$ times a group of lower unitriangular matrices, we
have that the character group of $K$ is the same as that of $Z$. The result now follows from \cite[Thm.~4]{Po}.

\noindent(ii).\ Let $\mc O_{\rm reg}$, $N$ and $K$ be as in (i). We first show that $D\cap\mc O_{\rm reg}$ is irreducible.
Let $\varphi$ be the morphism from Lemma~\ref{lem.varphi} and let $\ov d$ be the minor
$\Big(\bmat{2\vspace{-.1cm}\\\vdots\\n}|\bmat{1\vspace{-.1cm}\\\vdots\\n-1}\Big)$ on $\Mat_{n,n-1}$.
As we have seen, $d=\pm\ov d\circ\varphi$. Let $S=(s_{ij})_{ij}\in\GL_n$. Write $S^{-1}$ as a block matrix
$\bmat{S^{-1}_{11}&S^{-1}_{12}\\S^{-1}_{21}&S^{-1}_{22}}$ according to the partition
$\{\{1\},\{2,\ldots,n\}\}$ of the row and the column indices. Then $\pm d(S^{-1}NS)=\ov d(\varphi(S^{-1}NS))=\det(S_{22}^{-1}X)$,
where $X$ consists of the last $n-1$ rows of $(NSe_1,\ldots,N^{n-1}Se_1)$.
Now $\det(S_{22}^{-1})=s_{11}/\det(S)$ and $\det(X)=s_{11}^{n-1}$, so $d(S^{-1}NS)=0$ if and only if $s_{11}=0$.
So $D\cap\mc O_{\rm reg}$ is the image of an irreducible set under a morphism and therefore irreducible.
To prove that $D$ is irreducible it now suffices to show that $D\cap\mc O_{\rm reg}$ is dense in $D$ and for this it suffices to show
that the subregular nilpotent orbit is contained in the closure of $D\cap\mc O_{\rm reg}$. This is clear if $n=2$, since
$D\cap\mc O_{\rm reg}$ is a cone, so we assume now that $n\ge3$.

Let $N_{\rm sr}$ be the matrix which is zero at position $(n,n-1)$ and equal to the corresponding entry of $N$ elsewhere,
and put $$S_t=\begin{bmatrix}
0&\Nts\nts t^{n-1}\ \cdots&t^3&t^2&t\\
1&\cdots&0&0&0\\
\vdots&\ddots&\vdots&\vdots&\vdots\\
0&\cdots&1&0&0\\
0&\cdots&0&0&-t^{-1}
\end{bmatrix}.\text{\quad Then\ }
S^{-1}_t=\begin{bmatrix}
0&1&\cdots&0&0\\
\vdots&\vdots&\ddots&\vdots&\vdots\\
0&0&\ldots&1&0\\
t^{-2}&0&\Nts-t^{n-3}\cdots&-t&1\\
0&0&\cdots&0&-t
\end{bmatrix}
$$

$$\text{and\ }S_t^{-1}N S_t=\begin{bmatrix}
0&\Nts\Nts\Nts t^{n-1}\ \cdots&t^4&t^3&t^2&t\\
1&\cdots&0&0&0&0\\
\vdots&\ddots&\vdots&\vdots&\vdots&\vdots\\
0&\cdots&1&0&0&0\\
-t^{n-3}&\cdots&-t&1&0&0\\
0&\cdots&0&-t&0&0
\end{bmatrix}.
$$
So $S_t^{-1}NS_t\in D\cap\mc O_{\rm reg}$ for all $t\ne0$ and $\lim_{t\to0}S_t^{-1}NS_t$ exists and equals $N_{sr}$.
A simple computation shows that $\Stab_K(N_{sr})$ has dimension $2$ (see e.g. \cite[\S 3.1]{Jan2}), so $K\cdot N_{sr}$ has dimension $n^2-n-2$ and is therefore open
in the subregular nilpotent orbit. Since $D\cap\mc O_{\rm reg}$ is $K$-stable we get that the subregular
nilpotent orbit is contained in the closure of $D\cap\mc O_{\rm reg}$.

By the proof of Lemma~\ref{lem.varphi} we have $k[\mc N\sm D]=k[\mc N][d^{-1}]\cong k[\Mat_{n,n-1}][\ov d^{-1}]$ which is a unique factorisation domain.
So by \cite[Prop.~II.6.5(c)]{H2}, ${\rm CL}(\mc N)$ is generated by the class of $D$. The principal divisor $(d)$ is equal to $mD$ for some $m\ge1$.
Since the class of $D$ generates ${\rm CL}(\mc N)$, we must have $n|m$ by (i). On the other hand we have from our previous computations
$\pm d(S^{-1}NS)=s_{11}^n/\det(S)$. So the pull-back of $(d)$ along the orbit map $S\mapsto S^{-1}NS$ is $n$ times an irreducible divisor, and therefore $m\le n$.
A more direct proof which avoids the use of (i) goes as follows. The function $S\mapsto s_{11}/s_{12}$ is fixed under the left regular action of $K$ and
therefore descends to a regular function on $\mc O_{\rm reg}$.
So $\{S\in\GL_n\,|\,s_{11}=0\}$ occurs with coefficient $1$ in the pullback of the principal divisor of a rational function on $\mc O_{\rm reg}$
and is therefore equal to the pullback of $D$. It follows that $m=n$. Now one can deduce (i).

\noindent(iii).\ If $\chi=[(r),\lambda]$, then $\chi+\eta=[(r+n-1),\lambda+{\bf 1}_{n-1}]$, where ${\bf 1}_{n-1}$ is the all-one
vector of length $n-1$. Deleting the first column of a tableau is
a bijection from the semi-standard tableaux of shape $\lambda+{\bf 1}_{n-1}$ with entries in $\{1,\ldots,n-1\}$
to the semi-standard tableaux of shape $\lambda$ with entries in $\{1,\ldots,n-1\}$. So, by Theorem.~\ref{thm.basis}(i),
the two weight spaces have the same dimension.

\noindent(iv).\ If $f/d^t$, $f\in k[\mc N]$, is a $B$-semi-invariant of weight $\chi$, then $f$ is a $B$-semi-invariant of weight $\chi+t\eta$
which must be dominant and in the root lattice. So an $r$ as stated exists and we fix a choice for $r$. Now we may assume $t\ge r$.
Then $f/d^{t-r}\in k[\mc N]^U_{\chi+r\eta}$ by (iii). So $g\mapsto g/d^r$ is an isomorphism from $k[\mc N]^U_{\chi+r\eta}$
onto $k[\mc N][d^{-1}]^U_\chi$. The assertion now follows from Theorem.~\ref{thm.basis}(i).
\end{proof}

\begin{remgl}
If $T$ is a one-column tableau with entries in $\{1,\ldots,n-1\}$, then $u_T$, $\tilde u_T$, $v_T$ and $\tilde v_T$ have ``primitive" weights, i.e.
their weights are not the sum of two nonzero dominant weights in the root lattice. From this it easily follows that they are irreducible, see \cite[Lem.~3]{T}.
It follows from \cite[Lemme~2]{Moe} that the same holds for their restrictions to $\mc N$. In particular $d$ is irreducible.
Note that, by Proposition~\ref{prop.d}(i), $k[\mc N]$ is not a unique factorisation domain.
\end{remgl}

\noindent{\it Acknowledgement}. I would like to thank the referees for the suggestion to use the generalised Chevalley restriction map in the proof of Theorem~\ref{thm.matid} and for numerous corrections.

\bigskip

{\sc\noindent College of Engineering, Mathematics and Physical Sciences,\\
University of Exeter, EX4 4QF, Exeter, UK.\\
{\it E-mail address : }{\tt R.Tange@exeter.ac.uk}
}

\end{document}